\documentclass[12pt]{article}

\usepackage[T1]{fontenc}
\usepackage[cp1250]{inputenc}
\usepackage{amsmath,amsthm}
\usepackage{amssymb}
\usepackage{amsfonts}
\usepackage{multirow}
\usepackage{multicol}
\usepackage[a4paper,margin=2cm]{geometry}
\newtheorem{theorem}{Theorem}
\newtheorem{proposition}{Proposition}

\newtheorem{definition}{Definition}
\newtheorem{example}{Example}

\newtheorem{corollary}{Corollary}
\newtheorem{remark}{Remark}

\title{Stretched non-positive Weyl connections on solvable Lie groups}
\author{Maciej Boche\'nski, Piotr Jastrz\c ebski and Aleksy Tralle}

\begin{document}
\maketitle
\abstract{We determine the structure of solvable Lie groups endowed with invariant stretched non-positive Weyl connections and find classes of solvable Lie groups admitting and not admitting such connections. In dimension 4 we fully classify solvable Lie groups and  compact solvmanifolds which admit invariant SNP Weyl connections}
\vskip6pt
\noindent {\it Keywords:} invariant Weyl connection, solvable Lie algebra, solvmanifold, non-positive curvature
\vskip6pt
\noindent {\it AMS Subject Classification:}  22E25, 53C30
\section{Introduction}
In \cite{TW} the authors deal with Weyl connections. 
By definition \cite{F}, a Weyl connection $\hat\nabla$ on a manifold $M$ is determined by a 1-form $\varphi$ satisfying the equality
$$\hat\nabla_Xg=-2\varphi(X)g.$$
Since the difference between two connections is a tensor field, one derives the formula
$$\hat\nabla_XY=\nabla_XY-\varphi(Y)X-\varphi(X)Y+g(X,Y)E,$$
where $E$ is dual to $\varphi$ and $\nabla$ is the Levi-Civita connection. Weyl connections locally coincide with the Levi-Civita connections of the metrics in the same conformal class. It was shown by Gauduchon \cite{G} that for a compact manifold $M$ there exists a unique metric in the conformal class such that the respective vector field is divergent free. In this case, one cosiders the Gauduchon representation $(g,E)$ of the Weyl connection. 
Weyl connections are used in the study of Gaussian thermostats. A Gaussian thermostat is an ordinary differential equation on the tangent bundle to a Riemannian manifold of the form
\begin{equation} {dx\over dt}=v,\nabla_vv=E-{g(E,v)\over g(v,v)}v, 
\end{equation}\label{eq-therm}
where $x(t)$ is a parametrized curve in $M$, $v\in TM$.
The trajectories of the Gaussian thermostat are geodesics of a Weyl connection \cite{PW}, they are also called {\it thermostat geodesics} \cite{AR}. Gaussian thermostats were introduced in \cite{H} and were used (among other studies) in creating interesting models in nonequilibrium statistical mechanics. In these models, the vector field $E$ contributes to the kinetic energy, while the second term of the thermostat equation retains the kinetic energy to be constant. In this paper we treat Gaussian thermostats from the geometric point of view, as a geometric structure $(M,g,E)$ satisfying (\ref{eq-therm}).

In \cite{TW} the authors pose the problem of studying invariant Weyl connections on homogeneous spaces. One of the motivations for this study came from \cite{BR}, \cite{R} and \cite{W}. It was shown in \cite{W}, that if the curvature of the Weyl connection is negative, then the Gaussian thermostat has some kind of "hyperbolic properties" of the geodesic flow (see also \cite{PW}), and this was in accordance with models of nonequilibrium statistical mechanics and even more general dynamical properties, like {\it A-axiom} \cite{BR},\cite{R}.  Thus, looking for examples of such thermostats, it is natural to begin with the invariant case.
The method described in \cite{TW} is a s follows. Let $G/H$ be a homogeneous Riemannian manifold with an orthogonal decomposion $\mathfrak{g}=\mathfrak{h}+\mathfrak{p}$. Let 
$$\mathfrak{p}_0=\{X\in\mathfrak{p}\,| [Y,X]=0\,\forall Y\in\mathfrak{h}\}.$$
Note that if $\mathfrak{h}=0$, then $\mathfrak{p}_0=\mathfrak{g}$ and  this is  the case considered in \cite{TW} in greater detail. The projection of an invariant vector field  on $G$  determined by  $E\in\mathfrak{p}_0$ defines a Weyl connection on $G/H$ (see \cite{TW}). 
Note, however, that there are no known examples of Weyl connections  of {\it negative} curvature on compact manifolds $M$ of dimension $\geq 3$. On the other hand, Gaussian thermostats were studied on surfaces of negative curvature \cite{AR}. Therefore,  one looks for  conditions  with features of non-positivity.  One of such assumptions  is the {\it SNP condition}. 
\begin{definition} {\rm  Let $(M,g)$ be a Riemannian manifold. Let $E$ be a vector field and $\gamma E,\gamma>0$ be a family of Weyl connections defined by vector fields $\gamma E$. A vector field $E$ is called {\it streched non-positive} (SNP) if there exists $\gamma_0\geq 0$ such that the Weyl connections defined by the fields $\gamma E$ are non-positive for all $\gamma\geq \gamma_0$. A Weyl connection on a compact Riemannian manifold $M$ is called streched non-positive (SNP) if its Gauduchon representation $E$ is SNP.}
\end{definition}  
Note that the SNP condition is the property of the Gaussian thermostat itself.  It says that there is a non-isolated family  $(M,g,\gamma E)$ of Gaussian thermostats which yields non-positivity for the values of parameter $\gamma$ exceeding some $\gamma_0$. In this sense, the classification of homogeneous spaces with invariant SNP Weyl connections answers the question: {\it what are Gaussian thermostats of the form $(G/H,g,\gamma E)$ which are non-positive for $\gamma\geq\gamma_0$?}  
Some necessary conditions (for general Riemannian manifolds) which imply SNP are found in \cite{TW}, Propositions 3.1 and 3.2. Note that some of them are subject to the condition $\operatorname{div}E=0$, and some not.  We will use the following.
\begin{proposition}\label{prop:snp-gen} If a Weyl connection  determined by a divergence-free vector field $E$ on a compact manifold is SNP, then
\begin{itemize}
\item{\rm (W1)} $K(\Pi)\leq 0$ for every plane $\Pi$ containing $E\not=0$,
\item {\rm (W2)} $\langle\nabla_YE,Y\rangle=0$ for every $Y\perp E$, $E\not=0$,
\end{itemize}
\end{proposition}
In \cite{TW} the following problem is considered. Let $G$ be a unimodular Lie group with an invariant Riemannian metric $g$. Find invariant SNP connections on $G$, that is, find invariant vector fields $E\in\mathfrak{g}$ such that $(g,E)$ determines an invariant SNP Weyl connection. The following  is a summary of results obtained in this direction.
\begin{theorem}[\cite{TW}, Theorem 4.3]\label{thm:parallel} For a unimodular Lie group $G$, if a left-invariant vector field $E\in\mathfrak{g}$ satisfies (W1) and (W2), then $E$ is parallel.
\end{theorem}
Also, in \cite{TW} the authors looked for {\it isolated} examples  of Gaussian  thermostats $(G,g, E)$ with non-positive Weyl curvature.  The following cases were analyzed:
\begin{enumerate}
\item $\mathfrak{g}=\mathfrak{a}\oplus_{\varphi}\tilde{\mathfrak{a}}$, where $\mathfrak{a}$ and $\tilde{\mathfrak{a}}$ are abelian Lie algebras of dimesnions $n+1$ and $n$, respectively, and $\varphi:\mathfrak{a}\rightarrow \mathfrak{gl}(\tilde{\mathfrak{a}})$. The non-positivity of the Weyl connection determined by $E\in\mathfrak{a}$ is expressed in terms of the eigenvalues of $\varphi(E)$ (\cite{TW},Theorem 5.1).
\item  the case of a 3-dimensional Lie group, where  non-positive Weyl connections were found on the 3-dimensional solvable Lie group $\operatorname{Sol}^3$, which is semidirect product $\mathbb{R}\rtimes_{\varphi}\mathbb{R}^3$ determined by 
$$
\varphi(t)=
\begin{pmatrix}
e^t & 0\\
0 & e^{-t}
\end{pmatrix}
$$
\end{enumerate}
Note that 1) is a consequence  of   the results of \cite{TW}  which we summarize as follows.
\begin{theorem}\label{thm:tw-group} For a unimodular Lie group $G$ endowed with an invariant Riemannian metric $g$ a left-invariant vector field $E\in\mathfrak{g}$ is SNP if and only if $\operatorname{ad}\,E$ is skew-symmetric and $E\perp [\mathfrak{g},\mathfrak{g}]$.
\end{theorem}
\begin{proof} The proof follows from \cite{TW}, Proposition 4.1, Theorem \ref{thm:parallel} and a straightforward observation that a parallel vector field is SNP.
\end{proof} 
The 3-dimensional case 2) was analyzed using Milnor's description of left-invariant metrics on Lie groups \cite{M}.
In this note we  pose the following problems.
\begin{enumerate}
\item  Find all invariant SNP Weyl connections on unimodular Lie groups of dimension  4.
\item Classify all unimodular solvable Lie groups admitting invariant Weyl connections with SNP property (thus, all solvable Lie groups which can yield SNP Gaussian thermostats $(G,g,\gamma E)$).
\end{enumerate}
Note that the unimodularity condition is quite natural, because it is a necesary condition for a Lie group $G$ to admit a lattice, and, in particular a co-compact lattice $\Gamma$ \cite{Ra}. If $\Gamma$ does exist, we get compact homogeneous spaces $G/\Gamma$ with homogeneous SNP Weyl connections.   
In the subsequent sections we solve  problems 1) and 2). Note that we don't address here  the dynamics of the Gaussian thermostats (which presumably is simple and reduces to the behavior of the integral curves of $E$). Our aim is to solve the geometric problem of classifying homogeneous spaces admitting SNP Weyl connections. Also, we do not address the more diffiucult problems of finding homogeneous  spaces with Weyl connections of non-positive curvature, which have more interesting dynamics (\cite{W2}, Section 10). The main results of this article are Theorem \ref{thm:small-dim} and Theorem \ref{thm:full-e}. These theorems show that the SNP condition on the Gaussian thermostat $(G,g,E)$ imposes structural restrictions on $G$ (Theorem \ref{thm:full-e}). Also, we find all 
compact 4-dimensional solvmanifolds admitting an SNP Weyl connection (Corollary \ref{cor:solvmnfds}). We find classes of solvable Lie groups admitting and not admitting invariant SNP Weyl connections among solvable extensions of nilpotent Lie groups whose Lie algebras belong to some known classes of nilpotent Lie algebras: metabelian, of Vergne type $\{2n,1,1\}$, and characteristically nilpotent (Propositions \ref{prop:semidirect}, \ref{prop:(n,2)}, \ref{prop:(2n,1,1)}, \ref{prop:charact}).  

Finally we want to mention that the general problem of classifying homogeneous spaces with invariant Weyl connections of non-positive curvature seems to be geometrically and algebraically very attractive and challenging. On one hand, one encounters   the rigidity phenomena similar to the case of invariant Riemannian metrics on homogeneous spaces. Any simply connected homogeneous space endowed with an invariant Riemannian metric of non-positive curvature admits a simply transitive action of a solvable Lie group $S$ whose Lie algebra satisfies strong structural restrictions (such algebras are called NC-algebras), see  \cite{AW}. However, imposing the restriction of unimodularity, one necessarily obtains flatness of the metric. On the other hand, the non-positivity of invariant Weyl connections seems to be less rigid and the results of this paper yield some small evidence for this. Note, however, that we impose the solvablity restriction on $G$ from the very beginning. 

Our terminology and notation basically follows \cite{TW}. We use \cite{GK} as a general reference for the theory of solvable and nilpotent Lie algebras. In the proof of Theorem \ref{thm:small-dim} we use some facts of the theory of lattices in Lie groups which can be found in \cite{Ra}. The material related to invariant connections and Riemannian metrics on homogeneous spaces is contained in \cite{KN1}, \cite{KN2}.   The Lie algebras of Lie groups $G$, $H$,..., are denoted by $\mathfrak{g},\mathfrak{h}$,... . The universal cover of a Lie group $G$ is denoted by $\tilde G$. 
\vskip6pt
\noindent {\bf Acknowledgement.} We thank Maciej P. Wojtkowski for valuable discussions. The first named author acknowledges the support of the National Science Center, Poland (grant NCN no. 2018/31/D/ST1/00083). The third named author was supported by the National Science Center, Poland (grant  2018/31/B/ST1/00053).

\section{Weyl connections on solvable Lie groups of dimension 4}

Our purpose is to fully classify all 4-dimensional unimodular solvable Lie group admitting invariant SNP Weyl connections and also all possible pairs $(g,E)$ on $G$, that is, all possible invariant Weyl SNP connections on them. To do this, we use a classification of invariant Riemannian metrics of such groups from \cite{T}.
Let us begin with a short description of this classification which we need to formulate our results. Let $G$ be a simply connected Lie group. A left invariant Riemannian metric on $G$ determines a scalar product on $\mathfrak{g}$, and conversely, a scalar product $\langle -,-\rangle$ on $\mathfrak{g}$ determines a left-invariant Riemannian metric on $G$. Let $\tilde{\mathfrak{M}}$ be the set of scalar products on $G$. The group $\operatorname{Aut}(\mathfrak{g})$ acts on $\tilde{\mathfrak{M}}$ by the formula $A\cdot\langle v,w\rangle=\langle A^{-1}v,A^{-1}w\rangle$ for any $A\in\operatorname{Aut}(\mathfrak{g})$. It is straightforward to see that to classify the left invariant Riemannian metrics on $G$ is equivalent  to classify the scalar products on $\mathfrak{g}$, which is the same as to find representatives of the orbits of $\operatorname{Aut}(\mathfrak{g})$ in $\tilde{\mathfrak{M}}$.  This is the result of \cite{T}. Note that in general  this classifcation is finer than just up to isometry, because an isometry of $(G,g)$ need not be a group isomorphism. However, this does not influence our arguments. In this article, we present the list of all  simply connected 4-dimensional solvable Lie groups which admit an SNP pair $(g,E)$ and show a representative in the appropriate $\operatorname{Aut}(\mathfrak{g})$-orbit.  
 Our notation follows that of \cite{T}. Here is the list of all 4-dimensional unimodular solvable Lie groups.
\begin{enumerate}
\item Nilpotent Lie groups:
$$\mathbb{R}^4, \operatorname{Nil}\times\mathbb{R},\operatorname{Nil}^4.$$
\item Solvable Lie groups:
$$\operatorname{Sol}^4_{m,n},\operatorname{Sol}^3\times\mathbb{R},\operatorname{Sol}_0^4,\operatorname{Sol}'^4_0,
\operatorname{Sol}_{\mu}^4,\widetilde{\operatorname{Isom}_0(\mathbb{R}^2)}\times\mathbb{R},\operatorname{Sol}_l^4,
\widetilde{S^1\rtimes\operatorname{Nil}^3}.$$
\end{enumerate}
Groups $\operatorname{Nil}\times \mathbb{R}, \operatorname{Nil}^4, \operatorname{Sol}^4_{m,n}, \operatorname{Sol}_3\times\mathbb{R},\operatorname{Sol}_0^4, \operatorname{Sol}'^4_0,\operatorname{Sol}_{\mu}^4$ and $\widetilde{\operatorname{Isom}_0(\mathbb{R}^2)}\times\mathbb{R}$ are all of the form 
$$\mathbb{R}\rtimes_{\varphi}\mathbb{R}^3,$$ 
where $\varphi:\mathbb{R}\rightarrow \operatorname{GL}(3,\mathbb{R})$  is defined as follows:
\begin{itemize}
\item $\operatorname{Nil}\times\mathbb{R}$: 
$$\varphi(t)=
\begin{pmatrix} 1 & t & 0\\
0 & 1 & 0\\
0 & 0 & 1
\end{pmatrix}
$$
\item $\operatorname{Nil}^4:$
$$\varphi(t)=
\begin{pmatrix} 1 & t & \frac{1}{2} t^2\\
0 & 1 & t\\
0 & 0 & 1
\end{pmatrix}
$$

\item $\operatorname{Sol}_{m,n}^4$:
$$\varphi(t)=
\begin{pmatrix} e^{\lambda t} & 0 & 0\\
0 &  e^t & 0\\
0 & 0 & e^{-(1+\lambda)t}
\end{pmatrix}
$$
 
\item $\operatorname{Sol}^3\times\mathbb{R}$:
$$\varphi(t)=
\begin{pmatrix}
1 & 0 & 0\\
0 & e^t & 0\\
0 & 0 & e^{-t}
\end{pmatrix}
$$

\item $\operatorname{Sol}_0^4$:
$$
\varphi(t)=
\begin{pmatrix}
e^t & 0 & 0\\
0 & e^t & 0\\
0 & 0 & e^{-2t}
\end{pmatrix}
$$
\item $\operatorname{Sol}_0'^4$:
$$\varphi(t)=
\begin{pmatrix}
e^t & te^t & 0\\
0 & e^t & 0\\
0 & 0 & e^{-2t}
\end{pmatrix}
$$
\item $\operatorname{Sol}_{\mu}^4$:
$$\varphi(t)=
\begin{pmatrix}
e^{\mu t}\cos t & e^{\mu t}\sin t t & 0\\
-e^{\mu t}\sin t & e^{\mu t}\cos t & 0\\
0 & 0 & e^{-2\mu t}
\end{pmatrix}
$$
\item $\widetilde{\operatorname{Isom}_0(\mathbb{R}^2)}\times\mathbb{R}$:
$$\varphi(t)=
\begin{pmatrix} \cos t &\sin t & 0\\
-\sin t & \cos t & 0\\
0 & 0 & 1
\end{pmatrix}
$$

\end{itemize}
Groups $\operatorname{Sol}_l^4$ and $\widetilde{S^1\rtimes_{\varphi}\operatorname{Nil}}$ are described as semidirect products with the 3-dimensional group of unipotent triangular matrices. The action of $\varphi(t)$ is given as follows:
\begin{itemize}
\item $\operatorname{Sol}_1^4$: $\varphi(t)$ acts on the unipotent  matrix
$$\begin{pmatrix}
1 & x & z\\
0 & 1 & y\\
0 & 0 & 1
\end{pmatrix}
$$
by the formula
$$\begin{pmatrix}
1 & e^tx & z\\
0 & 1 & e^{-t}y\\
0 & 0 & 1
\end{pmatrix}
$$
\item $\widetilde{S^1\rtimes_{\varphi}\operatorname{Nil}}$: $\varphi(t)$-action is given by the formula
$$\begin{pmatrix}
1 & x\cos t +y\sin t & z+p(x,y,t)\\
0 & 1 & -x\sin t+y\cos t\\
0 & 0 & 1
\end{pmatrix}
$$
where $p(x,y,t)=\frac{1}{2}(y^2-x^2)\cos t\sin t-xy\sin^2 t$.
 \end{itemize}
Note that the Lie algebras corresponding to the Lie groups from the above table are well known and can be easily calculated. The standard Lie bracket structure on $\mathbb{R}\oplus_{\varphi}\mathbb{R}^3$ and on $\mathbb{R}\oplus_{\varphi}\mathfrak{n}_3$ is not convenient to analyze the SNP condition which is a condition on the sectional curvature. However, \cite{T} contains a classification of all invariant Riemannian metrics on the given Lie groups $G$  in the following sense: the scalar product $\langle -,-\rangle$ on $\mathfrak{g}$ is determined by expressing the orthonormal base $X_1,...,X_4$ ("Milnor base") through the standard base $e_1,...,e_4$. We do not reproduce the standard bases here, however, each of them is explicitly written in the course of the proof of Theorem \ref{thm:small-dim}.  Here we assume that $e_1,e_2,e_3$ span either $\mathbb{R}^3$ or $\mathfrak{n}_3$, where $\mathfrak{n}_3$ is the Lie algebra of the triangular $(3\times 3)$-matrices with zeros on the diagonal (the "3-dimensional Heisenberg algebra"). Note again, that in the formulation of Theorem \ref{thm:small-dim} we describe the representatives of the $\operatorname{Aut}(\mathfrak{g})$-orbits by writing down the family of Milnor bases. 

\begin{theorem}\label{thm:small-dim} A non-abelian solvable unimodular Lie group $G$ admits an SNP Weyl connection if an only if it is one of the following:
   $$\operatorname{Nil}\times\mathbb{R}, \widetilde{\operatorname{Isom}_0(\mathbb{R}^2)}\times\mathbb{R},\widetilde{\operatorname{Nil}\rtimes S^1}, \operatorname{Sol}^3\times\mathbb{R}.$$ 
Any of these groups admits a co-compact lattice $\Gamma$, so any of solvmanifolds $G/\Gamma$ is a compact 4-dimensional manifold with an SNP connection. The following list yields all possible families of pairs $(g,E)$ determining invariant SNP Weyl connections:
\begin{itemize}
\item $\operatorname{Nil}\times\mathbb{R}$, two-parameter family
$$\{be_1,e_2,e_3,e_4,b>0\},  E=\alpha e_2,$$
\item $\widetilde{\operatorname{Isom}_0(\mathbb{R}^2)}\times\mathbb{R}$, three-parameter family
$$\{e_1,be_2.e_3,e_3,ce_4, 0<b<1, c>0\}, E=\alpha e_3,$$
\item $\widetilde{\operatorname{Nil}\rtimes S^1}$ four-parameter family
$$\{ae_1,e_2,be_1+ce_3, de_4,a,b,d>0,0<c<1\}, E=\alpha e_4,$$
\item $\operatorname{Sol}^3\times\mathbb{R}$, three-parameter family
$$\{e_1,e_2,be_3+e_3,ce_4,b,c>0\},E=\alpha e_1.$$

\end{itemize}
\end{theorem}
\begin{proof} 
Any invariant Riemannian metric on $G$ is determined by the choice of the scalar product $\langle -,-\rangle$ on $\mathfrak{g}$.
By Theorem \ref{thm:parallel} if $(g,E)$ is invariant and SNP, then  $\nabla_YE=0$ for any invariant vector field $E$. Also, it satisfies (W2), which translates into the identity for the scalar product $\langle -,-\rangle$ on the Lie algebra $\mathfrak{g}$:  
\begin{equation}
\langle [E,Y],Y\rangle  = 0, \label{w2}
\end{equation}
for all $Y\in\mathfrak{g}$ orthogonal to $E$ (see the proof of Theorem 7.2 in \cite{TW}, or verify directly). 
The usual Riemannian geometry formula for $\nabla_YE$ yields
\begin{equation}
 \langle[Y,E],Z \rangle - \langle [Z,Y], E \rangle + \langle [Z,E] , Y \rangle =0 \label{nabla}
\end{equation}
for all $Y,Z\in\mathfrak{g}$.
We check these identities in each case separately, writing down the corresponding expressions for (\ref{w2}) and (\ref{nabla}) in terms of the Milnor bases. 
\vskip6pt
{\bf The case of nilpotent $\mathfrak{g}$}.
\vskip6pt


(I) $\mathfrak{g} = \operatorname{Lie}(\operatorname{Nil} \times \mathbb{R})$. From \cite{T}, Theorem 3.1, every metric is isometric to the one defined by the orthonormal base
$$X_1= b_{11} e_1, \ X_2=e_2, \ X_3=e_3, \ X_4 = e_4, \ b_{11} >0.$$
Calculating the Lie brackets we get 
$$[X_i, X_j] =0 \ \text{for} \ (i,j) \neq (3,4), (4,3)$$
$$[X_3, X_4] = \frac{1}{b_{11}} X_1, \ [X_4, X_3] = -\frac{1}{b_{11}} X_1.$$

Let $E=\sum_{i=1}^4 \alpha_i X_i, \ Y=\sum_{i=1}^4 \beta_i X_i$. Then

$$[E, Y]=\left[\alpha_{3} X_{3}, \beta_{4} X_{4}\right]+\left[\alpha_{4} X_{4}, \beta_{3} X_{3}\right] = \frac{\alpha_3\beta_4 - \alpha_4 \beta_3}{b_{11}} X_1$$

$$\langle [E,Y], Y \rangle = \left\langle \frac{\alpha_3\beta_4 - \alpha_4 \beta_3}{b_{11}} X_1 , \sum \beta_i X_i \right\rangle = \beta_1\frac{\alpha_3\beta_4 - \alpha_4 \beta_3}{b_{11}} X_1$$
Because $\langle [E,Y], Y \rangle=0$ for all $\beta_1, \beta_2, \beta_3, \beta_4$, we get $\alpha_3=\alpha_4=0.$ Therefore $E= \alpha_1 X_1 +\alpha_2X_2$ and  $ [E,Y]=0$ for all $Y\in\mathfrak{g}$. Then (\ref{w2}) implies $\langle [Z,Y], E \rangle=0$. By choosing the appropriate coefficients, we get $ [Z,Y] = \gamma X_1$ for all $\gamma\in\mathbb{R}$. Now we calculate
$$\langle [Z,Y], E \rangle = \langle \gamma X_1, \alpha_1 X_1 +\alpha_2X_2\rangle=\gamma\alpha_1=0.$$
Therefore, $a_1=0$ and $E= \alpha_2X_2$.

(II) $\mathfrak{g} = \operatorname{Lie}(\operatorname{Nil}^4)$. We write down the Lie brackets for the standard generators $e_1,...,e_4$ 
$$[e_1,e_4] = 0, \ [e_2,e_4]=e_1, \ [e_3,e_4]=e_2.$$
 By \cite{T}, Theorem 3.2, every metric is isometric to the one defined by an orthonormal base:
$$X_1= b_{11} e_1, \ X_2=b_{12}e_2+b_{22}e_2, \ X_3=e_3, \ X_4 = e_4, \ b_{11},b_{22} >0, \ b_{12} \geq 0.$$

We have 
$$[X_i, X_j] =0 \ \text{for} \ (i,j) \neq (2,4),(4,2), (3,4), (4,3)$$
$$[X_2, X_4] = \frac{b_{22}}{b_{11}}X_1 = \gamma_1 X_1 \ \text{for some} \ \gamma_1 = \frac{b_{22}}{b_{11}} >0$$
$$[X_3, X_4] = \frac{1}{b_{22}}X_2 -\frac{b_{12}}{b_{22}b_{11}} X_1 = \gamma_2X_2 -\gamma_3 X_1 \ \text{for some} \ \gamma_2>0, \gamma_3\geq 0.$$

Write down the expressions for $E$ and $Y$ in the form $E=\sum_{i=1}^4 \alpha_i X_i, \, Y=\sum_{i=1}^4 \beta_i X_i$. Then

$$[E,Y] = [(\alpha_2\beta_4-\alpha_4\beta_2)\gamma_1X_1+(\alpha_3\beta_4-\alpha_4\beta_3)(\gamma_2X_2 -\gamma_3 X_1)=$$
$$=\left((\alpha_2\beta_4-\alpha_4\beta_2)\gamma_1-(\alpha_3\beta_4-\alpha_4\beta_3)\gamma_3\right)X_1+(\alpha_3\beta_4-\alpha_4\beta_3)\gamma_2X_2$$

$$0 = \langle [E,Y] , Y \rangle = \left((\alpha_2\beta_4-\alpha_4\beta_2)\gamma_1-(\alpha_3\beta_4-\alpha_4\beta_3)\gamma_3\right) \beta_1 +(\alpha_3\beta_4-\alpha_4\beta_3) \beta_2 $$
for all $\beta_1,\beta_2, \beta_3, \beta_4 \in \mathbb{R}$. If $\beta_1=0, \beta_2=1, \beta_3=0, \beta_4=1$, we get $\alpha_3=0$. Analogously, it can be shown that $\alpha_4=0$ and $\alpha_2=0$. Thus, $E=\alpha_1 X_1$.  We obtain $ [E,Y]=0$ for all $Y\in\mathfrak{g}$. It follows  from (\ref{w2})  that $\langle [Z,Y], E \rangle=0$. By choosing the appropriate coefficients, we get $ [Z,Y] = \eta_1 X_1+\eta_2X_2$ for all $\eta_1,\eta_2\in\mathbb{R}$.This implies
$$\langle [Z,Y], E \rangle = \langle \eta_1 X_1+\eta_2X_2, \alpha_1 X_1 \rangle=\eta_1\alpha_1=0,$$
which yields $\alpha_1=0$ and $E= 0$.
\vskip6pt
\noindent {\bf  The case $\mathfrak{g}$  solvable with abelian nilradical}
\vskip6pt
 In this case $\mathfrak{g} =\mathbb{R}^3 \rtimes_{\varphi} \mathbb{R}$ and $\mathbb{R}e_4  \cong \mathbb{R}$. We also have 
$$\operatorname{ad} e_4 | _{\operatorname{Span}(e_1,e_2,e_3)} =\varphi(1) $$
 (from Table 4 in \cite{T}). By \cite{TW}, Theorem 4.1, $E$ is orthogonal to $[\mathfrak{g},\mathfrak{g}]$. We analyze the following subcases. 

(IIIA) $\mathfrak{g} = \operatorname{Lie}(\operatorname{Sol}_0^4)$. Write down the Lie brackets for the standard base:
$$[e_4,e_1] = e_1, \ [e_4,e_2] = e_2, \ [e_4,e_3] = -2e_3$$
By \cite{T}, Theorem 4.1, every metric is  defined by an orthonormal base:
$$X_1=  e_1, \ X_2=e_2, \ X_3=b_{13}e_1+e_3, \ X_4 = b_{44} e_4, \ b_{44} >0, \ b_{13}\geq 0.$$
We get
$$[X_4, X_1 ] = b_{44} X_1 $$
$$[X_4, X_2 ] = b_{44} X_2 $$
$$[X_4,X_3] =-2b_{44}X_3+3b_{44} b_{13} X_1$$
  Here $[\mathfrak{g},\mathfrak{g}]=\operatorname{Span}(X_1,X_2,X_3)$ , so $E=\alpha_4X_4$ for some $\alpha_4\in\mathbb{R}$. Let $Y=\beta_1X_1+\beta_2X_2+\beta_3X_3$. Then
$$[E,Y] =[\alpha_4X_4, \beta_1X_1+\beta_2X_2+\beta_3X_3 ]  = \alpha_4 b_{44} \left(
(\beta_1+3\beta_3b_{13})X_1 +\beta_2 X_2-2\beta_3X_3
\right)$$
$$0=\langle [E,Y] , Y\rangle = \alpha_4 b_{44} \left(
\beta_1^2+3\beta_1\beta_3b_{13} +\beta_2 ^2-2\beta_3^2\right)$$
But $b_{44}>0$. So $\alpha_{4}=0$ or $
\beta_1^2+3\beta_1\beta_3b_{13} +\beta_2 ^2-2\beta_3^2=0$ for all $\beta_1, \beta_2, \beta_3 \in \mathbb{R}$. Then $\alpha_{4}=0$, which implies $E=0$.

(IIIB) $\mathfrak{g} = \operatorname{Lie}(\operatorname{{Sol}'_0}^4)$. As in the previous cases
$$[e_4,e_1] = e_1, \ [e_4,e_2] = e_1+e_2, \ [e_4,e_3] = -2e_3.$$
\cite{T}, Theorem 4.2 implies that every metric is  defined by an orthonormal base:
$$X_1=  e_1, \ X_2=b_{22}e_2, \ X_3=b_{13}e_1+b_{23}e_2+e_3, \ X_4 = b_{44} e_4, \ b_{22},b_{44} >0, \ b_{13}\geq 0,\ b_{23}\in \mathbb{R}.$$
Calculating the brackets we get
$$[X_4, X_1 ] = b_{44} X_1 $$
$$[X_4, X_2 ] = b_{44} X_2 +b_{22}b_{44}X_1 = b_{44}(\gamma_1X_2+\gamma_2X_1) \ \text{for some} \ \gamma_1,\gamma_2>0$$
$$[X_4,X_3] = -2b_{44}X_3+3\frac{b_{23}b_{44}}{b_{22}}X_2+b_{44}(3b_{13}+b_{23})X_1 =  b_{44}(\gamma_3X_3+\gamma_4X_2 +\gamma_5X_1)$$
$$\ \text{for some} \ \gamma_3<0,\gamma_4,\gamma_5\in\mathbb{R}$$
 Here $[\mathfrak{g},\mathfrak{g}]=\operatorname{Span}(X_1,X_2,X_3)$ , so $E=\alpha_4X_4$ for some $\alpha_4\in\mathbb{R}$. Let $Y=\beta_1X_1+\beta_2X_2+\beta_3X_3$. Then
$$[E,Y] = \alpha_4\beta_1 b_{44} X_1 +\alpha_4\beta_2 b_{44} (\gamma_1X_2+\gamma_2X_1)+\alpha_4\beta_3 b_{44} (\gamma_3X_3+\gamma_4X_2 +\gamma_5X_1)$$
$$0=\langle [E,Y] , Y\rangle = \alpha_4\beta_1^2 b_{44} +\alpha_4\beta_2 b_{44}\gamma_2\beta_1 +\alpha_4\beta_3 b_{44}\gamma_5\beta_1+\alpha_4\beta_2 b_{44}\gamma_1\beta_2+\alpha_4\beta_3 b_{44}\gamma_4\beta_2+$$
$$+\alpha_4\beta_3^2 b_{44}\gamma_3 = \alpha_4 b_{44} (\beta_1^2  +\beta_2 \gamma_2\beta_1 +\beta_3 \gamma_5\beta_1+\beta_2^2 \gamma_1+\beta_3 \gamma_4\beta_2+\beta_3^2 \gamma_3)$$ for all $\beta_1, \beta_2, \beta_3 \in \mathbb{R}$. Then $\alpha_{4}=0$, and, therefore, $E=0$.

(IIIC) $\mathfrak{g} = \operatorname{Lie}(\operatorname{Sol}_{m,n}^4)$. As before
$$[e_4,e_1] = \lambda e_1, \ [e_4,e_2] = e_2, \ [e_4,e_3] = (-1-\lambda) e_3 \ \text{for some} \ \lambda >1.$$
By \cite{T}, Theorem 4.3  every metric is  defined by an orthonormal base
$$X_1=  e_1, \ X_2=b_{12}e_1+e_2, \ X_3=b_{13}e_1+b_{23}e_2+e_3, \ X_4 = b_{44} e_4, \ b_{44} >0, \ b_{12},b_{23}\geq 0,\ b_{13}\in \mathbb{R}.$$
We get the following expressions of the Lie brackets
$$[X_4, X_1 ] = b_{44} \lambda X_1 $$
$$[X_4, X_2 ] = b_{44} X_2 +(\lambda-1)b_{12}b_{44} X_1 = b_{44}(X_2+\gamma_1X_1) \ \text{for some} \ \gamma_1\geq0$$
$$[X_4,X_3] = b_{44}(\gamma_2X_3+\gamma_3X_2 +\gamma_4X_1) \ \text{for some} \ \gamma_2<0,\gamma_3,\gamma_4\in\mathbb{R}$$
Because $\lambda>1>0$, this case is analogous to (IVB). We obtain $E=0$.

(IIID) $\mathfrak{g} = \operatorname{Lie}(\operatorname{Sol}^3\times\mathbb{R})$. We have
$$[e_4,e_1] = 0, \ [e_4,e_2] = e_2, \ [e_4,e_3] = -e_3.$$
\cite{T}, Theorem 4.4 shows that every metric is  defined by an orthonormal basis
$$X_1=  e_1, \ X_2=b_{12}e_1+e_2, \ X_3=b_{13}e_1+b_{23}e_2+e_3, \ X_4 = b_{44} e_4, \ b_{44} >0, \ b_{12},b_{23}\geq 0,\ b_{13}\in \mathbb{R}.$$
We obtain the equalities
$$[X_4, X_1 ] = 0 $$
$$[X_4, X_2 ] = b_{44} X_2 -b_{12}b_{44} X_1 = b_{44}(X_2+\gamma_1X_1) \ \text{for} \ \gamma_1 = -b_{12}\leq0$$
$$[X_4,X_3] = b_{44}(\gamma_2X_3+\gamma_3X_2 +\gamma_4X_1) $$
$$\ \text{for } \ \gamma_2=-1<0,\gamma_3=2b_{23} \geq0,\gamma_4=b_{13}-2b_{23}b_{12}\in\mathbb{R}$$
Let us consider two subcases. The first one arises  if  $[\mathfrak{g},\mathfrak{g}]=\operatorname{Span}(X_2,X_3)$. Then $b_{12}=b_{13}=0$. Put $E=\alpha_1X_1+\alpha_4X_4$ for some $\alpha_1, \alpha_4\in\mathbb{R}$ and $Y=\beta_1X_1+\beta_2X_2+\beta_3X_3+\beta_4X_4$. Here
$$[E,Y] =\alpha_4\beta_2 b_{44} X_2+\alpha_4\beta_3 b_{44} (\gamma_2X_3+\gamma_3X_2 ), $$
$$0=\langle [E,Y],Y \rangle = \alpha_4\beta_2 b_{44} \beta_2+\alpha_4\beta_3 b_{44} (\gamma_2\beta_3+\gamma_3\beta_2 )= b_{44} \alpha_4(\beta_2^2+\beta_3^2\gamma_2+\beta_3 \gamma_3\beta_2)$$
for all $\beta_1, \beta_2, \beta_3, \beta_4 \in \mathbb{R}$. It follows that $\alpha_{4}=0$. Then $E=\alpha_1X_1$. We obtain $ [E,Y]=0$ for all $Y\in\mathfrak{g}$ and from (\ref{nabla})  $\langle [Z,Y], E \rangle=0$. By choosing the appropriate coefficients, we get $ [Z,Y] = \eta_2X_2+\eta_3X_3$ for all $\eta_2,\eta_3\in\mathbb{R}$.Then we obtain
$$\langle [Z,Y], E \rangle = \langle \eta_2X_2+\eta_3X_3, \alpha_1 X_1 \rangle==0.$$
Finally $E=\alpha_1X_1$.

The second subcase arises when  $[\mathfrak{g},\mathfrak{g}]=\operatorname{Span}(X_1, X_2,X_3)$.
Put $Y=\beta_1X_1+\beta_2X_2+\beta_3X_3$. So $E=\alpha_4X_4$ for some $\alpha_4\in\mathbb{R}$. Then
$$[E,Y] = \alpha_4\beta_2 b_{44} (X_2+\gamma_1X_1)+\alpha_4\beta_3 b_{44} (\gamma_2X_3+\gamma_3X_2 +\gamma_4X_1)$$
$$0=\langle [E,Y] , Y\rangle = \alpha_4\beta_2 b_{44}\gamma_1\beta_1 +\alpha_4\beta_3 b_{44}\gamma_4\beta_1+\alpha_4\beta_2 b_{44}\beta_2+\alpha_4\beta_3 b_{44}\gamma_3\beta_2+$$
$$+\alpha_4\beta_3^2 b_{44}\gamma_2 = \alpha_4 b_{44} (\beta_2 \gamma_1\beta_1 +\beta_3 \gamma_4\beta_1+\beta_2 ^2+\beta_3 \gamma_3\beta_2+\beta_3^2 \gamma_2)$$ for all $\beta_1, \beta_2, \beta_3 \in \mathbb{R}$. Then $\alpha_{4}=0$. Finally $E=0$.

(IIIE) $\mathfrak{g} = \operatorname{Lie}(\operatorname{Sol}_{\mu}^4)$. We get $E=0$ in a way similar to  (IVC).

(IIIF) $\mathfrak{g} =\operatorname{Lie}( \widetilde{\operatorname{Isom}(\mathbb{R}^2) \times\mathbb{R}})$. Here $[e_i,e_j]=0$ for $i,j=1,2,3,$ and
$$[e_4,e_1]=-e_2, [e_4,e_2]=e_1, [e_4,e_3]=0.$$
By \cite{T}, Theorem 4.6 there are two forms of the corresponding orthonormal bases. 
We will write down the first one, the second is analogous.
$$X_1=  e_1, \ X_2=b_{22}e_2, \ X_3=b_{13}e_1+b_{23}e_2+e_3, \ X_4 = b_{44} e_4, \ b_{44} >0, \ b_{13},b_{23}\geq 0, 0<b_{22}<1.$$ We have:
$$[X_4,X_1] = -\frac{b_{44}}{b_{22}}X_2 = b_{44}\gamma_1 X_2  \ \text{for} \ \gamma_1 = -\frac{1}{b_{22}}\leq0 $$
$$[X_4,X_2] = b_{44}b_{22} X_1 = b_{44}\gamma_2 X_1\ \text{for} \ \gamma_2= b_{22} >0 $$
$$[X_4,X_3] = b_{44}b_{23}X_1 -\frac{b_{44}b_{13}}{b_{22}} X_2 =  b_{44} (\gamma_3 X_1+\gamma_4X_2)\ \text{for} \ \gamma_3= b_{23} \geq 0 , \gamma_4= -\frac{b_{13}}{b_{22}} \leq 0 $$
Here $[\mathfrak{g},\mathfrak{g}]=\operatorname{Span}(X_1,X_2)$ , so $E=\alpha_3X_3+\alpha_4X_4$ for some $\alpha_3,\alpha_4\in\mathbb{R}$. Let $Y=\beta_1X_1+\beta_2X_2+\beta_3X_3+\beta_4X_4$. Then:
$$[E,Y] = \alpha_4\beta_1 b_{44}\gamma_1 X_2 +\alpha_4\beta_2  b_{44}\gamma_2 X_1+(\alpha_4\beta_3-\alpha_3\beta_4)  b_{44} (\gamma_3 X_1+\gamma_4X_2)$$
$$0=\langle [E,Y] , Y\rangle =\alpha_4\beta_1 b_{44}\gamma_1 \beta_2 +\alpha_4\beta_2  b_{44}\gamma_2 \beta_1+(\alpha_4\beta_3-\alpha_3\beta_4)  b_{44} (\gamma_3 \beta_1+\gamma_4\beta_2)$$
for all $\beta_1,\beta_2, \beta_3, \beta_4 \in \mathbb{R}$. If $\beta_1=1, \beta_2=1, \beta_3=0, \beta_4=0$, we have $$0=\alpha_4 b_{44}\gamma_1  +\alpha_4  b_{44}\gamma_2 = \alpha_4b_{44} (\gamma_1+\gamma_2) = \alpha_4b_{44} \left(-\frac{1}{b_{22}} +b_{22} \right)=\alpha_4b_{44} \frac{b_{22}^2-1}{b_{22}}.$$ So $\alpha_4=0$. Then
$$0=\langle [E,Y] , Y\rangle = -\alpha_3\beta_4b_{44} (\gamma_3 \beta_1+\gamma_4\beta_2)$$
So $\alpha_3=0$ or $\gamma_3 \beta_1+\gamma_4\beta_2=0$. If $\alpha_3=0,$
 then $E=0$.
However, if $\gamma_3 \beta_1+\gamma_4\beta_2=0$, take $\beta_1=\beta_2=1$ and get $b_{13}=b_{23}=0$. So $E=\alpha_3 X_3$.  We obtain $ [E,Y]=0$ for all $Y\in\mathfrak{g}$, and from (\ref{w2})  $\langle [Z,Y], E \rangle=0$. By choosing the appropriate coefficients, we get $ [Z,Y] = \eta_1 X_1+\eta_2X_2$ for all $\eta_1,\eta_2\in\mathbb{R}$.Then we have
$$\langle [Z,Y], E \rangle = \langle \eta_1 X_1+\eta_2X_2, \alpha_3 X_3 \rangle=0.$$
Finally $E=\alpha_3 X_3 \neq 0$ if $b_{13}=b_{23}=0$.
\vskip6pt
{\bf The case of  solvable  $\mathfrak{g}$ with a non-abelian nilradical.}
\vskip6pt

(IVA) $\mathfrak{g} = \operatorname{Lie}(\operatorname{Sol}_1^4)$. The standard Lie brackets are written as follows.
$$[e_1, e_i] = 0 \ \text{for all} \ i$$
$$[e_2,e_3] =e_1, \ [e_4.e_2]=e_2, \ [e_4, e_3] =-e_3$$
\cite{T},Theorem 5.1 implies that every metric is  defined by an orthonormal basis:
$$X_1=  b_{11}e_1, \ X_2=b_{12}e_1+e_2, \ X_3=b_{13}e_1+b_{23}e_2+e_3, \ X_4 = b_{44} e_4, \ b_{11},b_{44} >0, \ b_{12},b_{23}\geq 0,\ b_{13}\in \mathbb{R}.$$
We have
$$[X_1, X_i] = 0 \ \text{for all} \ i$$
$$[X_2, X_3 ] =\gamma_1 X_1 \ \text{for some} \ \gamma_1>0$$
$$[X_2, X_4 ] =\gamma_2 X_2+\gamma_3 X_1 \ \text{for some} \ \gamma_2<0, \gamma_3\in\mathbb{R}$$
$$[X_3, X_4 ] =\gamma_4 X_3+\gamma_5 X_2 +\gamma_6 X_1\ \text{for some} \ \gamma_4>0, \gamma_5,\gamma_6\in\mathbb{R}$$
Let $E=\sum \alpha_i X_i, \ Y=\sum \beta_i X_i$. Then
$$[E,Y] =\left[\alpha_{3} X_{3}, \beta_{4} X_{4}\right]+\left[\alpha_{4} X_{4}, \beta_{3} X_{3}\right]+ \left[\alpha_{2} X_{2}, \beta_{4} X_{4}\right]+\left[\alpha_{4} X_{4}, \beta_{2} X_{2}\right]+\left[\alpha_{3} X_{3}, \beta_{2} X_{2}\right]+\left[\alpha_{2} X_{2}, \beta_{3} X_{3}\right]$$
If we write down the equation $0=\langle [E,Y] , Y\rangle$ and choose the appropriate values $\beta_i$ we will obtain $\alpha_2=\alpha_3=\alpha_4=0$. Therefore $E=\alpha_1X_1$. We obtain $ [E,Y]=0$ for all $Y\in\mathfrak{g}$ and  (\ref{nabla})  implies $\langle [Z,Y], E \rangle=0$. By choosing the appropriate coefficients, we get $ [Z,Y] = \eta_1 X_1+\eta_2X_2+\eta_3X_3$ for all $\eta_1,\eta_2,\eta_3\in\mathbb{R}$.Then we obtain
$$\langle [Z,Y], E \rangle = \langle \eta_1 X_1+\eta_2X_2+\eta_3X_3, \alpha_1 X_1 \rangle=\eta_1\alpha_1=0.$$
This shows that  $\alpha_1=0$ and $E= 0$.

(IVB) $\mathfrak{g} = \operatorname{Lie}(\widetilde{\operatorname{Nil}\rtimes S^1})$. We have
$$[e_1, e_i] = 0 \ \text{for all} \ i$$
$$[e_2,e_3] =e_1, \ [e_4.e_2]=-e_3, \ [e_4, e_3] =e_2$$
By \cite{T}, Theorem 5.2  there are two forms of the corresponding orthonormal bases. The first one  has the form
$$X_1=b_{11}e_1, \ X_2=b_{12}e_1+e_2, \  X_3=b_{13}e_1+b_{33}e_3, \ X_4 =b_{44}e_{4}, \ b_{11}, b_{44}>0, 0<b_{33}<1, b_{12}b_{13}\geq 0$$
Then
$$[X_1, X_i] = 0 \ \text{for all} \ i$$
$$[X_2, X_3 ] =\gamma_1 X_1 \ \text{for} \ \gamma_1=\frac{b_{33}}{b_{11}}>0$$
$$[X_2, X_4 ] =\gamma_2 X_3+\gamma_3 X_1 \ \text{for} \ \gamma_2=-\frac{b_{44}}{b_{33}}<0, \gamma_3=\frac{b_{44}b_{13}}{b_{33}b_{11}}\geq0$$
$$[X_3, X_4 ] =\gamma_4 X_2 +\gamma_5 X_1\ \text{for} \ \gamma_4=b_{33}b_{44}>0, \gamma_5=-\frac{b_{33}b_{44}b_{12}}{b_{11}}\leq0$$
Write $E=\sum_{i=1}^4 \alpha_i X_i, \ Y=\sum_{i=1}^4 \beta_i X_i$. Then
$$[E,Y] =\left[\alpha_{3} X_{3}, \beta_{4} X_{4}\right]+\left[\alpha_{4} X_{4}, \beta_{3} X_{3}\right]+ \left[\alpha_{2} X_{2}, \beta_{4} X_{4}\right]+\left[\alpha_{4} X_{4}, \beta_{2} X_{2}\right]+\left[\alpha_{3} X_{3}, \beta_{2} X_{2}\right]+\left[\alpha_{2} X_{2}, \beta_{3} X_{3}\right] = $$
$$=(\alpha_3\beta_4-\alpha_4\beta_3)(\gamma_4 X_2 +\gamma_5 X_1) +(\alpha_2\beta_4-\alpha_4\beta_2)(\gamma_2 X_3+\gamma_3 X_1)+(\alpha_2\beta_3-\alpha_3\beta_2) \gamma_1 X_1$$
$$0=\langle [E,Y] , Y\rangle=(\alpha_3\beta_4-\alpha_4\beta_3)(\gamma_4 \beta_2 +\gamma_5 \beta_1) +(\alpha_2\beta_4-\alpha_4\beta_2)(\gamma_2 \beta_3+\gamma_3 \beta_1)+(\alpha_2\beta_3-\alpha_3\beta_2) \gamma_1 \beta_1$$
Take $\beta_1=\beta_3=0, \beta_2=\beta_4=1$, so $\alpha_3\gamma_4=0$. But $\gamma_4>0$, then $\alpha_3=0$. Therefore
$$0=\langle [E,Y] , Y\rangle = -\alpha_4\beta_3(\gamma_4 \beta_2 +\gamma_5 \beta_1) +(\alpha_2\beta_4-\alpha_4\beta_2)(\gamma_2 \beta_3+\gamma_3 \beta_1)+\alpha_2\beta_3 \gamma_1 \beta_1$$
Take $\beta_1=\beta_2=0, \beta_3=\beta_4=1$, we have $\alpha_2\gamma_2=0$, but $\gamma_2<0$. then $\alpha_2=0$. So
$$0=\langle [E,Y] , Y\rangle = -\alpha_4\beta_3(\gamma_4 \beta_2 +\gamma_5 \beta_1) -\alpha_4\beta_2(\gamma_2 \beta_3+\gamma_3 \beta_1)=$$
$$=-\alpha_4(\beta_3\gamma_4\beta_2+\beta_3\gamma_5\beta_1+\beta_2\gamma_2\beta_3+\beta_2\gamma_3\beta_1)$$
for all $\beta_1,\beta_2, \beta_3, \beta_4 \in \mathbb{R}$. So $\alpha_4=0$ or $\beta_3\gamma_4\beta_2+\beta_3\gamma_5\beta_1+\beta_2\gamma_2\beta_3+\beta_2\gamma_3\beta_1=0$. In this second subcases after changing the signs of $\gamma_i$, we obtain $\gamma_4=-\gamma_2$ and $\gamma_3=\gamma_5=0$. Because $\gamma_4\neq-\gamma_2$, one obtains $\alpha_4=0$, and $E=\alpha_1X_1$. We get $ [E,Y]=0$ for all $Y\in\mathfrak{g}$ and from (\ref{nabla})  one obtains $\langle [Z,Y], E \rangle=0$. By choosing the appropriate coefficients, we get $ [Z,Y] = \eta_1 X_1+\eta_2X_2+\eta_3X_3$ for all $\eta_1,\eta_2,\eta_3\in\mathbb{R}$.Then we have
$$\langle [Z,Y], E \rangle = \langle \eta_1 X_1+\eta_2X_2+\eta_3X_3, \alpha_1 X_1 \rangle=\eta_1\alpha_1=0.$$
Finally $\alpha_1=0$ and $E= 0$.

Consider the second form of the  orthonormal basis:
$$X_1=b_{11}e_1, \ X_2=b_{12}e_1+e_2, \  X_3=e_3, \ X_4 =b_{44}e_{4}, \ b_{11}, b_{44}>0, b_{12}\geq 0$$
The calculation is similar provided we put $b_{13}=0$ and $b_{33}=1$. We get
$$0=\langle [E,Y] , Y\rangle = -\alpha_4(\beta_3\gamma_4\beta_2+\beta_3\gamma_5\beta_1+\beta_2\gamma_2\beta_3)$$ where $\gamma_4=b_{44}>0$, $\gamma_2=-b_{44}<0, \gamma_5=-\frac{b_{44}b_{12}}{b_{11}} \leq0$. Because $\gamma2=-\gamma_4$, we have:
$$0=\langle [E,Y] , Y\rangle = -\alpha_4\beta_3\gamma_5\beta_1$$
So $\alpha_4=0$, or $\gamma_5=0$. If $\alpha_4=0$, we again get $E=0$. 

If $\gamma_5=0$, then $b_{12}=0$ and $E=\alpha_1X_1+\alpha_4X_4$. Weobtain $ [E,Y]=0$ for all $Y\in\mathfrak{g}$ and then (\ref{nabla})  implies $\langle [Z,Y], E \rangle=0$. By choosing the appropriate coefficients, we get $ [Z,Y] = \eta_1 X_1+\eta_2X_2+\eta_3X_3$ for all $\eta_1,\eta_2,\eta_3\in\mathbb{R}$.Then we have
$$\langle [Z,Y], E \rangle = \langle \eta_1 X_1+\eta_2X_2+\eta_3X_3, \alpha_1 X_1+\alpha_4X_4 \rangle=\eta_1\alpha_1=0.$$
So $\alpha_1=0$ and  $E=\alpha_4X_4$ if $b_{12}=0$.

It remains to prove that each of the Lie groups admitting an SNP Weyl connection admits a lattice. We will do this by constructing it in each case. 
\begin{itemize}
\item The group $\operatorname{Nil}\times\mathbb{R}$ admits a co-compact lattice, because the coresponding Lie algebra $\mathfrak{n}\oplus\mathbb{R}$ clearly admits a $\mathbb{Q}$-structure (this is a general existence criterion for lattices in nilpotent Lie groups).
\item The case of  $\widetilde{\operatorname{Isom}_0(\mathbb{R}^2)}\times\mathbb{R}$ is settled as follows. This group has the form of a semidirect product $\mathbb{R}^3\rtimes_{\varphi}\mathbb{R}$. One can see that one can form a discrete subgroup $\mathbb{Z}^3\rtimes_{\varphi}\pi\mathbb{Z}$, which is a co-compact lattice. 
\item If $G=\widetilde{S^1\rtimes_{\varphi}\operatorname{Nil}}$, then form a subgroup $\Gamma=\Gamma_1\rtimes_{\varphi}\pi\mathbb{Z}$, where $\Gamma_1$ is any finitely generated subgroup in $N_3(\mathbb{Q})$ of rank 3. 
\item Consider the case $\operatorname{Sol}^3\times\mathbb{R}$. One can easily see that this group can be also presented in the form $\mathbb{R}\times G_1$, where $G_1$ is a 3-dimensional solvable Lie group of the form $G_1=\mathbb{R}^2\rtimes_{\phi}\mathbb{R}$, where $\phi$ is a one-parameter subgroup of the form
$$\phi(t)=
\begin{pmatrix}
e^t & 0\\
0 & e^{-t}
\end{pmatrix}
$$
It known that $G_1$ admits a co-compact lattice $\Gamma_1$ (the full proof can be found in \cite{TO}, Theorem 1.9). Thus, $\mathbb{Z}\times\Gamma_1$ is the required co-compact lattice in $G$. 

\end{itemize}
\end{proof}
It is known that in dimension 4 all compact solvmanifolds are of the form $G/\Gamma$, where $\Gamma$ is a co-compact lattice in a four-dimensional solvable Lie group $G$ \cite{GO}, Theorem 3.1. This yields the following.
\begin{corollary}\label{cor:solvmnfds} All 4-dimensional compact solvmanifolds with invariant SNP Weyl connections are exhaused by the list given in Theorem \ref{thm:small-dim}.
\end{corollary}

\section{A structure of solvable Lie groups admitting SNP Weyl connections}
Here we determine the structure all solvable Lie algebras such that the corresponding  Lie groups admit invariant SNP Weyl connections. 
\begin{theorem}\label{thm:full-e}
 Unimodular Lie group $G$ admits an invariant SNP Weyl connection determined by a non-central $E\in\mathfrak{g}$  if and only if $\mathfrak{g}$ has the form
$$\mathfrak{g}=\langle E\rangle\oplus_{\varphi}\mathfrak{s},$$
where:
\begin{enumerate}
\item $\mathfrak{s}$ is a unimodular solvable Lie algebra such that $\operatorname{Aut}(\mathfrak{s})$ contains a compact torus $T$ of positive dimension,
\item $\varphi: \langle E\rangle\rightarrow\operatorname{Der}(\mathfrak{s})$ has image in $\mathfrak{t}\subset\operatorname{Der}(\mathfrak{s)}$, that is $\varphi(\langle E\rangle)\subset\mathfrak{t}$,
\item $\mathfrak{n}:=[\mathfrak{s},\mathfrak{s}]=[\mathfrak{g},\mathfrak{g}]$.
\end{enumerate}  
\end{theorem}
\begin{proof}
Choose any scalar product $\langle-,-\rangle$ on $\mathfrak{g}$ and assume that $E$ determines an SNP Weyl connection. By Theorem \ref{thm:tw-group}, $E\perp [\mathfrak{g},\mathfrak{g}]=\mathfrak{n}$ and so $\mathfrak{s}:=<E>^{\perp}$ is an ideal of $\mathfrak{g} .$ Also the action of $\operatorname{ad}_{E}$ on $\mathfrak{n}$ is skew-symmetric. It follows that $\operatorname{ad}_{tE}\in \operatorname{Der}(\mathfrak{n})$ corresponds to a subgroup $\tilde{T}$ in $\operatorname{Aut}(\mathfrak{n})$ such that the closure $T$ of $\tilde{T}$ is a compact torus.
Assume that there is no non-trivial compact torus in $\operatorname{Aut}(\mathfrak{n})$ and so $\operatorname{ad}_{E}|_{\mathfrak{n}} = 0.$ Consider the orthogonal decomposition
$$\mathfrak{g}=\mathfrak{a}\oplus \mathfrak{n},$$
$E\in \mathfrak{a} .$ By \cite{TW}, Theorem 4.3 we have
$$\forall _{Y,Z\in\mathfrak{g}} \ \langle[Y,E],Z \rangle - \langle [E,Z], Y \rangle + \langle [Z,Y] , E \rangle =0 .$$
Thus for any $A\in\mathfrak{a}$ and any $Z\in \mathfrak{n}$
$$\langle[A,E],Z \rangle - \langle [E,Z], A \rangle + \langle [Z,A] , E \rangle = \langle[A,E],Z \rangle =0. $$
It follows that $[E,A]=0$ for any $A\in\mathfrak{a}$ and so $E$ is central in $\mathfrak{g} .$
On the other hand if $\operatorname{ad}_{E}|_{\mathfrak{n}} \neq 0$ and $T$ is a compact torus of positive dimension. 

If $\mathfrak{g}=\langle E\rangle\oplus_{\varphi}\mathfrak{s}$ and satisfies 1)-3), then by Theorem \ref{thm:tw-group} $\mathfrak{g}$ admits an invariant SNP Weyl connection.
\end{proof}
\section{Classes of solvable Lie groups admitting invariant SNP Weyl connections}
We simplify the general description given by Theorem \ref{thm:full-e} as follows.
\begin{proposition}\label{prop:semidirect} Assume that $\mathfrak{n}$ is a nilpotent Lie algebra with the following property: $\operatorname{Aut}(N)$ contains a compact subgroup $S$. Let $T\subset S$ be a torus. There exists a semidirect product $\mathfrak{g}=\mathfrak{a}\oplus_{\varphi}\mathfrak{n}$ and a scalar product $\langle-,-\rangle$ on $\mathfrak{g}$ such that for any $E\in\mathfrak{a}$ the pair $(\langle-,-\rangle,E)$ determines an SNP Weyl connection.
\end{proposition}
\begin{proof} Consider the Lie algebra  $\operatorname{Der}(\mathfrak{n})$ of the Lie group $\operatorname{Aut}(N)$. Let $\mathfrak{s}$ be the Lie algebra of $S\subset\operatorname{Aut}(N)$, and $\mathfrak{t}$ be the Lie algebra of $T$. Since $\mathfrak{t}$ is abelian,  for any abelian Lie algebra $\mathfrak{a}$ and any linear map $\varphi:\mathfrak{a}\rightarrow\mathfrak{t}\subset\mathfrak{s}\subset\operatorname{Der}(\mathfrak{n})$ one can construct a semidirect sum 
$$\mathfrak{g}=\mathfrak{a}\oplus_{\varphi}\mathfrak{n}.$$
Define a $T$-invariant scalar product $\langle-,-\rangle_n$ on $\mathfrak{n}$ and take any scalar product on $\mathfrak{a}$. Declairing $\mathfrak{a}$ and $\mathfrak{n}$ orthogonal, we get a scalar product on $\mathfrak{g}$ which determines an invariant Riemannian metric on $\mathfrak{g}$. Then, by the construction, $E\perp[\mathfrak{g},\mathfrak{g}]$, and $\operatorname{ad}\,E|_{\mathfrak{n}}=\varphi(E)\in\mathfrak{t}$ is skew symmetric.

\end{proof}

Proposition \ref{prop:semidirect} yields series of examples of solvable Lie groups with SNP Weyl connections in any dimension.
\subsection{Lie algebras of some particular Vergne's types}
\begin{definition}[\cite{V}]\label{def:vergne} The Vergne  type $\{d_1,...,d_r\}$ {\rm of a nilpotent Lie algebra $\mathfrak{n}$ with descending central series 
$\mathfrak{n}^{(i)}=
[\mathfrak{n},\mathfrak{n}^{(i-1)}]$ 
is defined by}
$$d_i=\dim\,(\mathfrak{g}^{(i-1)}/\mathfrak{g}^{(i)}).$$
\end{definition}
\noindent In particular, we can characterize a subclass of  the class of nilpotent Lie algebras of type $\{n,2\}$ as follows.
\begin{definition} {\rm  Nilpotent Lie algebras of type {\it $\{n,2\})$-Heisenberg}  are nilpotent Lie algebras $V\oplus\langle x,y\rangle$ of dimension $n+2$ defined by a pair of alternating forms $F_1$ and $F_2$ on the $n$-dimensional vector space $V$ putting for any $v,w\in V, [v,w]=F_1(v,w)x+F_2(v,w)y$.}
\end{definition}

\begin{proposition}\label{prop:(n,2)} Any unimodular solvable  Lie group which is a semidirect product $A\rtimes_{\varphi}N$ of an abelian Lie group $A$ and a nilpotent Lie group $N$ whose Lie algebra $\mathfrak{n}$ has type $\{n,2\}$-Heisenberg, admits an SNP Weyl connection.
\end{proposition}
\begin{proof} In view of Proposition \ref{prop:semidirect} it is sufficient to prove that $\operatorname{Aut}(\mathfrak{n})$ has a compact subgroup. This can be seen as follows. Denote by $Sp(V, F_1)$ a subgroup of $GL(V)$ which preserves the alternating form $F_1$. This group naturally embeds into $\operatorname{Aut}(\mathfrak{n})$ by the formula
$$f|_{V\oplus\langle x,y\rangle}=f|_V\oplus\operatorname{id}_{\langle x,y\rangle}.$$
Since $Sp(F_1)$ is a closed subgroup of $GL(V)$, it is a Lie group, containing $Sp(2l)$, where $2l\leq n$ is the rank of $F_1$. Thus, it contains a non-trivial maximal compact subgroup, as required.
\end{proof}
\begin{remark} {\rm A full classification of nilpotent Lie algebras of type $\{n,2\}$ whose automorphism group contains a compact torus is obtained in \cite{FF}. We do not reproduce it here, since the description of such Lie algebras is rather complicated.}
\end{remark}

\begin{proposition}\label{prop:(2n,1,1)} Any unimodular semidirect product $A\rtimes N$ of an abelian Lie group and a realification of a complex nilpotent Lie group of type $\{2n,1,1\}$ admits an SNP Weyl connection.
\end{proposition}
\begin{proof} The proof follows from Proposition \ref{prop:semidirect}. It is known that under the assumptions of the Proposition, $\operatorname{Aut}(N)$ contains a compact subgroup. In greater detail, let $\mathfrak{n}$ be a complex nilpotent Lie algebra of type $\{2n,1,1\}$. Then, by \cite{BBF} the Lie algebra $\operatorname{Der}(\mathfrak{n})$ has the form $\operatorname{Der}(\mathfrak{n})=\mathfrak{r}\oplus\mathfrak{sp}(n,\mathbb{C})$, where $\mathfrak{r}$ denotes the radical of $\operatorname{Der}(\mathfrak{n})$. It follows that for the realification (denoted by the same letter), the automorphism group contains a compact torus. 
\end{proof}
  
\begin{example} {\rm A $(2n+1)$-dimensional Heisenberg Lie algebra is defined as a a vector space $V\oplus\langle x\rangle$ of dimension $2n+1$ with the only  non-zero Lie brackets of the form $[v,w]=F(v,w)x$, where $F$ is a symplectic form on $V$. As in the proof of Proposition \ref{prop:(n,2)} one can see that $\operatorname{Aut}(\mathfrak{n})$ contains $Sp(n,\mathbb{R})$. Thus, any semidirect product of an abelian Lie group with the Heisenberg Lie group admits an SNP Weyl connection.}
\end{example}

\subsection{Metabelian Lie algebras} {\rm A finite dimensional Lie algebra $\mathfrak{g}$ is called {\it metabelian}, if $[\mathfrak{g},[\mathfrak{g},\mathfrak{g}]]=0$. The {\it signature} of a metabelian Lie algebra is a pair $(m,n)$, where $m=\dim \mathfrak{g}/[\mathfrak{g},\mathfrak{g}]$, $n=\dim [\mathfrak{g},\mathfrak{g}]$. Note that this is a particular case of Definition \ref{def:vergne}.  A metabelian Lie algebra structure on $\mathfrak{g}$ is completely determined by the commutator map $\Lambda^2U\rightarrow  V$, where $V=[\mathfrak{g},\mathfrak{g}]$ and $U$ is a complement in $\mathfrak{g}$ (different complements determine different structures). Conversely, let $\mathfrak{g}=U\oplus V$ be a direct sum of two vector spaces $U$ and $V$ of dimensions $m$ and $n$. Then each skew symmetric bilinear surjective map $f:\Lambda^2U\rightarrow V$ determines a metabelian Lie algebra structure on $\mathfrak{g}$ of signature $(m,n)$. The space of maps $f$ is $\Lambda^2U^*\otimes V$, and the group $GL(U)\times GL(V)$ naturally acts on it.  Thus, a classification of metabelian Lie algebras can be understood  in terms of  the orbits of this group. In \cite{GT} the canonical elements $f$  determining  the orbits are found.} We will call such classification a  Galitskii-Timashev classification and the corresponding families of Lie algebras {\it the Galitskii-Timashev types}. Note that these types are determined over $\mathbb{C}$ and we condier algebraic tori in algebraic groups.
\begin{proposition}[\cite{GT}, Section 1.1]\label{metab} Let $\mathfrak{g}$ be a metabealian complex  Lie algebra determined by $f:\Lambda^2U\rightarrow V$. Then 
$$\operatorname{Aut}(\mathfrak{g})=G(f)\rtimes N$$
where $N$ is a unipotent subgroup, and $G(f)$ denotes the $GL(U)\times GL(V)$-stabilizer.
\end{proposition}
{\rm Thus, if the connected component of $G(f)$ of a metabelian Lie algebra $\mathfrak{n}$ contains an algebraic torus,  any semidirect sum $\mathfrak{g}=\mathfrak{a}\oplus_{\varphi}\mathfrak{n}$ given by Proposition \ref{prop:semidirect} admits an SNP invariant connection, provided that $\mathfrak{n}$ is a realification (denoted also as $\mathfrak{n}$) . The paper \cite{GT} contains examples of metabealian Lie algebras of both types, containing an algebraic  torus and not containing it. In greater detail we can describe the classification as follows. 
Treat $f\in \Lambda^2U^*\otimes V$ as a tensor. Choose bases $_{1}e,...,_{m}e$ of $U$ and $e_1,...,e_n$ of $V$, and write  the base of $\Lambda^2U^*\otimes V$ in the form
$$^{ij}e_k=e^i\wedge e^j\otimes e_k.$$ 
Note that dual elements are denoted by raising or lowering the indices. The tables in \cite{GT} contain the description of $f$ in the dual form, so  the following notation is used:
$$(abc...ijk) \,\,\,\text{stands for}\,\,\,_{ab}e_c+\cdots+_{ij}e_k.$$
Analysing Tables 1-8 in \cite{GT} we get a description of canonical elements $f$ and the connected components of their stabilizers. By what we have said the following holds.
\begin{proposition}\label{prop:(5,5)} Any metabelian Lie algebra over $\mathbb{C}$ of signature $(m,n)$ such that $m,n\leq 5$, or $m\leq 6,n\leq 3$ has an automorphism group which contains an algebraic  torus with the following exceptions determined by the canonical choice of tensor $f$:
\vskip6pt
for $m,n\leq 5$:
\begin{center}
\begin{tabular}{|c|} \hline
$132 \; 521 \; 415 \; 354$\\\hline
$125 \; 144 \; 153 \; 234 \; 243 \; 252 \; 342 \; 351$\\\hline
$125 \; 134 \; 153 \; 233 \; 243 \; 252 \; 342 \; 451$\\\hline
$125 \; 135 \; 144 \; 152 \; 234 \; 242 \; 251 \; 343$\\\hline
$125 \; 134 \; 143 \; 152 \; 233 \; 244 \; 342 \; 451$\\\hline
$125 \; 143 \; 154 \; 233 \; 242 \; 251 \; 341 \; 352$\\\hline
$125 \; 132 \; 144 \; 153 \; 234 \; 243 \; 252 \; 351$\\\hline
$125 \; 134 \; 141 \; 153 \; 243 \; 252 \; 342 \; 351$\\\hline
$121 \; 144 \; 153 \; 234 \; 243 \; 252 \; 342 \; 451$\\\hline
$125 \; 134 \; 143 \; 152 \; 233 \; 242 \; 251 \; 341$\\\hline
\end{tabular}
\end{center}
for $m\leq 6, n\leq 3$:
\begin{center}
\begin{tabular}{|c|} \hline
$531  \; 152  \; 313$\\\hline
$121  \; 342  \; 531  \; 152  \; 313$\\\hline
$143  \; 162  \; 233  \; 252  \; 351$\\\hline
$143  \; 162  \; 233  \; 252  \; 261  \; 342  \; 351$\\\hline
$153  \; 162  \; 233  \; 242  \; 252  \; 261  \; 341$\\\hline
$133  \; 152  \; 161  \; 243  \; 252  \; 342  \; 351$\\\hline
$143  \; 161  \; 233  \; 242  \; 251  \; 341  \; 352 $\\\hline
$133  \; 142  \; 153  \; 161  \; 243  \; 252  \; 341$\\\hline
$123  \; 141  \; 152  \; 242  \; 261  \; 351  \; 362$\\\hline
$143  \; 152  \; 161  \; 233  \; 242  \; 251  \; 341$\\\hline
\end{tabular}
\end{center}
In total, there are 223  Galitskii-Timashev classes of  metabelian Lie algebras and 20 of them do not have an algebraic  torus in the automorphism group. 
\end{proposition}

 }  

\begin{corollary}\label{cor:metab} Let $\mathfrak{n}$ be a realification of a complex metabelian Lie algebra of signature $(m,n)$, $m,n\leq 5$ or $m\leq 6,n\leq 3$ such that its automorphism group contains an algebraic torus according to Proposition \ref{prop:(5,5)}. Then any solvable Lie group whose Lie algebra is a semidirect extension  as in Proposition \ref{prop:semidirect} admits an SNP Weyl connection.
\end{corollary} 
\begin{remark} {\rm  A classfication of metabelian Lie algebras in terms of generators and relations is obtained in \cite{Ga}. It would be interesting to find more general result than that of Proposition \ref{prop:(5,5)}.}
\end{remark}

\subsection{Solvable Lie groups which do not admit invariant SNP connections}
 Recall that a characteristically nilpotent Lie algebra is a nilpotent Lie algebra $\mathfrak{n}$ such that $\operatorname{Der}(\mathfrak{n})$ is nilpotent \cite{GK}. 
\begin{definition} {\rm W say that a nilpotent Lie algebra $\mathfrak{n}$ is {\it characteristically nilpotent of Dyer type}, if $\operatorname{Aut}(\mathfrak{n})$ is unipotent.}
\end{definition}
\begin{remark} {\rm An example of a characteristically nilpotent Lie algebra $\mathfrak{n}$ with unipotent $\operatorname{Aut}(N)$ was found by Dyer \cite{D}. More examples can be found in \cite{AC}. Dyer's example is a follows: it is a nine-dimensional Lie algebra spanned by $X_1,...,X_9$ with the Lie brackets}
$$[X_1,X_2]=X_3; \quad [X_1,X_3]=X_4; \quad  [X_1,X_5]=X_7; \quad [X_1,X_8]=X_9;$$
$$[X_2,X_3]=X_5; \quad [X_2,X_4]=X_7; \quad [X_2,X_5]=X_6; \quad [X_2,X_7]=-X_8;$$
$$[X_3,X_7]=-[X_4,X_5]=X_9.$$

\end{remark} 
\begin{proposition}\label{prop:charact} No solvable Lie group whose Lie algebra  is a semidirect product of an abelian and characteristically  nilpotent Lie algebra of Dyer type admits an invariant SNP Weyl connection.
\end{proposition}
\begin{proof} This is a corollary to Theorem \ref{thm:full-e}.
\end{proof}

\vskip10pt
Faculty of Mathematics and Computer Science,
\vskip10pt
\noindent University of Warmia and Mazury,
\vskip10pt
\noindent S\l\/oneczna 54, 10-710 Olsztyn, Poland
\vskip10pt
e-mail addresses:
\vskip6pt
mabo@matman.uwm.edu.pl (MB),
\vskip6pt
 piojas@matman.uwm.edu.pl (PJ),
\vskip6pt
 tralle@matman.uwm.edu.pl (AT)


\begin{thebibliography}{ABCDE}
\bibitem[AC]{AC}J. M. Anconchea, R. Campoamor, {\it Characteristically nilpotent Lie algebras: a survey}, Extracta Math. 16(2001), 153-210
\bibitem[AR]{AR} Y. M. Assylbekov, F. T. Rea, {\it The attenuated ray transforms on Gaussian thermostats with negative curvature}, arXiv:2102.04571
\bibitem[AW]{AW} R. Azencott, E. Wilson, {\it Homogeneous manifolds with negative curvature. I}, Trans. Amer. Math. Soc. 215(1976), 323-362.
\bibitem[BBF]{BBF} C. Bartolone, A. Di Bartolo, G. Falcone, {\it Solvable extensions of nilpotent complex Lie algebras of type $\{2n,1,1\}$}, Moscow Math. J. 18 (2018), 607-616.
\bibitem[BR]{BR} R. Bowen, D. Ruelle, {\it The ergodic theory of Axiom A flows}, Invent. Math. 29(1975), 181-202.
\bibitem[D]{D} J. L. Dyer, {\it A nilpotent Lie algebra with nilpotent automorphism group}, Bull. Amer. Math. Soc. 76(1970), 52-56.
\bibitem[F]{F} G. B. Folland, {\it Weyl manifolds}, J. Diff. Geom. 4(1970), 145-153
\bibitem[FF]{FF} G. Falcone, \'A. Figula, {\it The action of a compact Lie group on nilpotent Lie algebras of type $\{n,2\}$}, Forum. Math. 28(2016), 795-806.
\bibitem[G]{G} P. Gauduchon, {\it La 1-forme de torsion d'une variete hermitenne compacte}, Math. Ann. 267(1984), 495-518
\bibitem[GT]{GT} L. Yu. Galitskii, D. A. Timashev, {\it On classification of metabelian Lie algebras,} J. Lie Theory  9 (1999), 125-156.
\bibitem[Ga]{Ga} M. A. Gauger, {\it On the classification of metabelian Lie algebras}, Trans. Amer. Math. Soc. 179 (1973), 293-329.
\bibitem[GK]{GK} M. Goze, Y. Khakimjanov, {\it Nilpotent Lie algebras}, Springer, Berlin, 2014.
\bibitem[GO]{GO} V. V. Gorbatsevich, A. L. Onishchik, {\it Transformation groups}, Encycl. Math. Sci. 20, Springer, Berlin, 1993
\bibitem[H]{H} W. G. Hoover, {\it Molecular Dynamics}, Lect. Notes Phys. Springer,  258(1986).
\bibitem[KN1]{KN1} S. Kobayashi, K. Nomizu, {\it Foundations of Differential Geometry}, vols. 1, Interscience, New York, 1961.
\bibitem[KN2]{KN2} S. Kobayashi, K. Nomizu, {\it Foundations of Differential Geometry}, vol. 2, Interscience, New York,1969.
\bibitem[M]{M} J. Milnor, {\it Curvatures of left-invariant metrics on Lie groups}, Adv. Math. 21 (1976), 293-329.
\bibitem[PW]{PW} P. Przytycki, M. P. Wojtkowski, {\it Gaussian thermostats as geodesic flows of non-symmetric linear connections}, Commun. Math. Phys. 277(2008), 759-769.
\bibitem[Ra]{Ra} M. S. Raghunatan, {\it Discrete Subgroups of Lie Groups}, Springer, Berlin, 1972 
\bibitem[R]{R} D. Ruelle, {\it Smooth dynamics and new theoretical ideas in nonequilibrium statistical mechanics}, J. Statist. Phys. 95(1999), 393-468.
\bibitem[TW]{TW} G. Tereszkiewicz, M. P. Wojtkowski,  {\it Homogeneous Weyl connections of non-positive curvature}, Annals Golbal Anal. Geom. 51 (2017), 209-229.
\bibitem[T]{T} S. van Thuong, {\it Metrics on 4-dimensional unimodular lie groups}, Ann.  Global Anal.  Geom., 51 (2) (2016), 109-128. 
\bibitem[TO]{TO} A. Tralle, J. Oprea, {\it Symplectic manifold with no K\"ahler structure}, Springer, Berlin, 1997.

\bibitem[V]{V} M. Vergne, {\it Cohomologie des algebres de Lie nilpotentes. Application a l'etude de la vari\'et\'e des algebres de Lie nilpotentes}, Bull. Soc. Math. France 98 (1970), 81-116
\bibitem[W]{W} M. P. Wojtkowski, {\it W-flows on Weyl manifolds and Gaussian thermostats}, J. Math. Pure Apl. 79(2000), 953-974.
\bibitem[W2]{W2} M. P. Wojtkowski, {\it Integrability via reversibility}, J. Geom. Phys. 115(2017), 61-74. 
\end{thebibliography}
\end{document}